\documentclass[11pt]{article}
\usepackage{amsmath}
\usepackage{amsfonts}
\usepackage{amssymb}
\usepackage{amsthm}
\usepackage{breqn}
\usepackage{setspace}
\usepackage{fullpage}
\usepackage{enumitem}
\usepackage{comment}
\usepackage{hyperref}
\usepackage{bbm}
\usepackage{tikz}
\usepackage{mathtools} 
\usetikzlibrary{patterns,arrows,decorations.pathreplacing}

\newtheorem{lemma}{Lemma}
\newtheorem{theorem}{Theorem} 
 
\newtheorem{definition}{Definition}
\newtheorem{claim}{Claim}
\newtheorem{conjecture}{Conjecture}

\newtheorem{remark}{Remark}

\newcommand{\B}{\mathcal{B}}

\newcommand{\F}{\mathcal{F}}

\newcommand{\h}{\mathcal{H}}

\newcommand{\abs}[1]{\left\lvert{#1}\right\rvert}
\newcommand{\floor}[1]{\left\lfloor{#1}\right\rfloor}

\DeclareMathOperator{\ex}{ex}

\usepackage{authblk}

\begin{document}
\title{Tur\'an numbers of Berge trees}

\author[1,2]{Ervin Gy\H{o}ri} 
\author[1,2]{Nika Salia}
\author[3,4]{Casey Tompkins}
\author[2,5]{Oscar Zamora} 

\affil[1]{Alfr\'ed R\'enyi Institute of Mathematics, Hungarian Academy of Sciences. \newline
 \texttt{gyori.ervin@renyi.mta.hu, nika@renyi.hu}}
\affil[2]{Central European University, Budapest.\par
 \texttt{oscar.zamoraluna@ucr.ac.cr}     }
 \affil[3] {Discrete Mathematics Group, Institute for Basic Science (IBS), Daejeon, Republic of Korea. \par
 \texttt{ctompkins496@gmail.com}}
\affil[4] {Karlsruhe Institute of Technology, Germany.
  }
 \affil[5]{Universidad de Costa Rica, San Jos\'e.}

\maketitle
\begin{abstract}

A classical conjecture of Erd\H{o}s and S\'os asks to determine the Tur\'an number of a tree.  We consider variants of this problem in the settings of hypergraphs and multi-hypergraphs.
In particular, for all $k$ and $r$, with $r \ge k (k-2)$, we show that any $r$-uniform hypergraph $\h$ with more than $\frac{n(k-1)}{r+1}$ hyperedges contains a Berge copy of any tree with $k$ edges different from the $k$-edge star. This bound is sharp when $r+1$ divides $n$ and for such values of $n$ we determine the extremal hypergraphs.

\end{abstract}

\section{Background}

We recall a classic theorem of Erd\H{o}s and Gallai~\cite{erdHos1959maximal}. 
\begin{theorem}[Erd\H{o}s, Gallai~\cite{erdHos1959maximal}]

Let $n$, $k$ be positive integers and let $G$ be an $n$-vertex graph containing no path of $k$ edges, then 
\[
e(G) \le \frac{(k-1)n}{2}.
\]
Equality is obtained if and only if $k$ divides $n$ and $G$ is the graph consisting of $n/k$ disjoint complete graphs of size $k$.  
\end{theorem}

Erd\H{o}s and S\'os~\cite{erdos1984some} conjectured that the same bound would hold for any tree with $k$ edges. A proof of this conjecture for sufficiently large $k$ was announced in the 90's by Ajtai, Koml\'os, Simonovits and Szemer\'edi.  We will consider a variant of this problem in the setting of hypergraphs and multi-hypergraphs. We obtain exact results for the case of large uniformity.  

Given a hypergraph $\h$, we denote the vertex and edge sets of $\h$ by $V(\h)$ and $E(\h)$, respectively. We denote the number of vertices and hyperedges by $v(\h) = \abs{V(\h)}$ and $e(\h) = \abs{E(\h)} $. A hypergraph is said to be $r$-uniform if all of its hyperedges have size $r$. We now provide some definitions which we will need.

\begin{definition}
For a given uniformity $r$ and a fixed graph $G$, an $r$-uniform multi-hypergraph $\h$ is a  \emph{Berge copy} of G, if there exists an injection $f_1: V(G) \to V(\h)$ and a bijection $f_2:E(G) \to E(\h)$, such that if $e = \{v_{1},v_{2}\}\in E(G)$, then $\{f_1(v_{1}),f_1(v_{2})\} \subseteq f_2(e)$.  The set of Berge copies of $G$ is denoted by $\B G$. The sets $f_1(V(G))$ and $f_2(E(G))$ are called the \emph{defining} vertices and hyperedges, respectively.
\end{definition}

We recall the classical definition of the Tur\'an number of a family of hypergraphs. 

\begin{definition}
The Tur\'an number of a family of $r$-uniform hypergraphs $\F$, denoted $\ex_r(n,\F)$, is the maximum number of hyperedges in an $n$-vertex, $r$-uniform, simple-hypergraph which does not contain an isomorphic copy of $\h$, for all $\h \in \F$,  as a sub-hypergraph.  
\end{definition}

The same question may be asked for multi-hypergraphs, we denote  the Tur\'an number for multi-hypergraphs by $\ex_r^{multi}(n,\F)$. 


\begin{remark}
If every hypergraph in $\F$ has at least $r+1$ vertices, then $\ex_r^{multi}(n,\F)$ is infinite, since a hypergraph on $r$ vertices and  multiple copies of the same hyperedge  is $\F$-free. 
\end{remark}

The classical theorem of Erd\H{o}s and Gallai was extended to Berge paths in $r$-uniform hypergraphs by Gy\H{o}ri, Katona and Lemons~\cite{gyorikatonalemons}. 

\begin{theorem}[Gy\H{o}ri, Katona, Lemons \cite{gyorikatonalemons}] \label{gkl}

	Let $n,k,r$ be positive integers and let $\h$ be an $r$-uniform hypergraph with no Berge path of length $k$.  If $k>r+1>3$, we have
	\begin{displaymath}
	e(\h) \le \frac{n}{k} \binom{k}{r}.
	\end{displaymath}
	If $r \ge k>2$, we have
	\begin{displaymath}
	e(\h) \le \frac{n(k-1)}{r+1}.
	\end{displaymath}
\end{theorem}

The remaining case when $k = r + 1$ was settled later by Davoodi, Gy\H{o}ri, Methuku and Tompkins~\cite{davoodi}, the Tur\'an number matches the upper bound of Theorem~\ref{gkl} in the $k>r+1$ case.



We now turn our attention to the case of trees in hypergraphs. The Tur\'an number of certain kinds of trees in $r$-uniform hypergraphs has long been a major topic of research.  For example, there is a notoriously difficult conjecture of Kalai \cite{kalai} which is more general than the Erd\H{o}s-S\'os conjecture.  The trees which Kalai considers are generalizations of the notion of tight paths in hypergraphs.  In another direction, F\"uredi \cite{fur} investigated linear trees, constructed by adding $r-2$ new vertices to every edge in a (graph) tree. In this setting, he proved asymptotic results for all uniformities at least $4$. Whereas, the articles above considered classes of trees containing tight and linear paths, respectively, we will consider the setting of Berge trees.     

In the range when $k > r$, a number of results on forbidding Berge trees were obtained by Gerbner, Methuku and Palmer in~\cite{bigk}. In particular they proved that if we assume the Erd\H{o}s-S\'os conjecture holds for a tree $T$ with $k$ edges and all of its sub-trees and also that  $k>r+1$, we have $\ex_r(n,\B T) \le \frac{n}{k}\binom{k}{r}$ (a construction matching this bound when $k$ divides $n$ is given by $n/k$ disjoint copies of the complete $r$-uniform hypergraph on $k$ vertices). In the present paper, we will consider the range $r>k$, where we prove some exact results.

\section{Main Results}

Considering  multi-hypergraphs, we prove the following.
\begin{theorem}
\label{Multi_tree_theorem}
Let $n,k,r$ be positive integers and let $T$ be a $k$-edge tree, then for all $r\geq (k-1)(k-2)$,
\begin{displaymath}
                      \ex_r^{multi}(n,\B T)\leq \frac{n(k-1)}{r}.
\end{displaymath}

If $r > (k-1)(k-2)$ and $T$ is not a star, equality holds if and only if $r$ divides $n$ and the extremal multi-hypergraph is  $\frac{n}{r}$ disjoint hyperedges, each with multiplicity $k-1$.  If $T$ is a star equality holds only for all  $(k-1)$-regular multi-hypergraphs.
\end{theorem}

We conjecture that Theorem \ref{Multi_tree_theorem} holds for the following wider set of parameters.

\begin{conjecture}
\label{treeconj}
Let $n,k,r$ be positive integers and let $T$ be a $k$-edge tree, then for all $r \ge k+1$,
\begin{displaymath}
                      \ex_r^{multi}(n,\B T)\leq \frac{n(k-1)}{r}.
\end{displaymath} 
For all trees $T$, where $T$ is  not a star, equality holds if and only if $r$ divides $n$ and  the extremal multi-hypergraph is $\frac{n}{r}$ disjoint hyperedges each with multiplicity $k-1$.   
\end{conjecture}

The special case of Conjecture \ref{treeconj}, when the forbidden tree is a path, was settled by Gy\H{o}ri, Lemons, Salia and Zamora~\cite{new} (see the first corollary).

We now define a class of hypergraphs which we will need when we classify the extremal examples in our main result about simple hypergraphs, Theorem~\ref{Not_Star_Theorem}.

\begin{definition}
An $r$-uniform hypergraph $\h$ is  \emph{two-sided} if $V(\h)$ can be partitioned  into a set $X$ and pairwise disjoint sets $A_i$, $i=1,2,\dots,t$ (also disjoint from $X$) of size $r-1$, such that every hyperedge is of the form $\{x\} \cup A_i$ for some $x \in X$. We say that a two-sided $r$-uniform hypergraph is $(a,b)$-regular if every vertex of $X$ has degree $a$ and every vertex of  $\displaystyle\bigcup_{i=1}^{t}A_i$ has degree $b$.
\end{definition}

\begin{remark} A two-sided $r$-uniform hypergraph can also be viewed as a graph obtained by taking a bipartite graph $G$ with bipartite classes $X$ and $Y$, and  ``blowing up" each vertex of $Y$ to a set of size $r-1$, and replacing each edge $\{x,y\}$ by the $r$-hyperedge containing $x$ together with the blown up set for $y$.

\end{remark}

\begin{theorem}
\label{Not_Star_Theorem}
Let $n,k,r$ be positive integers and let $T$ be a $k$-edge tree which is not a star, then for all $r\geq k(k-2)$, 

\begin{displaymath}
                      \ex_r(n,\B T)\leq \frac{n(k-1)}{r+1}.
\end{displaymath}
Equality holds if and only if $r+1$ divides $n$, and the extremal hypergraph is obtained from $\frac{n}{r+1}$ disjoint sets of size $r+1$, each containing $k-1$ hyperedges. 
Unless $k$ is odd, and $T$ is the balanced double star, where the balanced double star is the tree obtain from and edge by adding $\frac{k-1}{2}$ incident edges to each of the ends of the edge, in which case equality holds if and only if $r+1$ divides $n$ and $\h$ is obtained from the disjoint union of sets of size $r+1$ containing $k-1$ hyperedges each and possibly a $(k-1,\frac{k-1}{2})$-regular two-sided $r$-uniform hypergraph (see Figure~\ref{Extremal}).
\end{theorem}

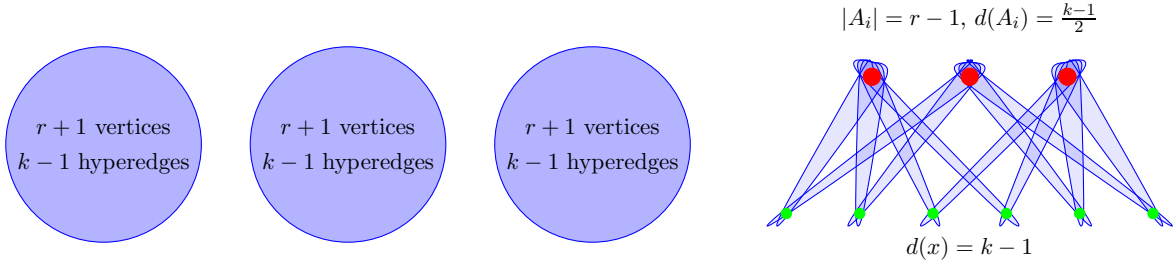
\begin{figure}[t]
    \centering
\begin{tikzpicture}[scale = 0.65,every node/.style={scale=.8}]
\draw (0,3);

\filldraw[blue,fill opacity=0.3] (2,0)  arc (0:360: 2cm and 2cm);
\draw (0,0) node[above]{$r+1$ vertices};
\draw (0,0) node[below] {$k-1$ hyperedges};


\begin{scope}[xshift=5cm]
\filldraw[blue,fill opacity=0.3] (2,0) arc (0:360: 2cm and 2cm);
\draw (0,0) node[above]{$r+1$ vertices};
\draw (0,0) node[below] {$k-1$ hyperedges};
\end{scope}

\begin{scope}[xshift=10cm]
\filldraw[blue,fill opacity=0.3] (2,0) arc (0:360: 2cm and 2cm);
\draw (0,0) node[above]{$r+1$ vertices};
\draw (0,0) node[below] {$k-1$ hyperedges};

\end{scope}
\draw(0,-2.5);
\end{tikzpicture} \qquad
\begin{tikzpicture}[rotate=90,scale = 0.65,every node/.style={scale=.8}]

\filldraw[blue, fill opacity=0.10] plot [smooth cycle] coordinates {  (0.02,3.68) (.04,3.83)  (2.85,2.25) (3.1,2) (2.70,1.75) } ;

\filldraw[blue, fill opacity=0.10] plot [smooth cycle] coordinates {  (0.02,2.18) (.04,2.33)  (2.85,2.25) (3.1,2) (2.70,1.75) } ;

\filldraw[blue, fill opacity=0.10] plot [smooth cycle] coordinates {  (0.02,.68) (.04,.83)  (2.85,2.25) (3.1,2) (2.70,1.75) } ;

\filldraw[blue, fill opacity=0.10] plot [smooth cycle] coordinates {  (0.02,-0.82) (.04,-0.6)  (2.85,2.25) (3.1,2) (2.70,1.75) } ;

\filldraw[blue, fill opacity=0.10] plot [smooth cycle] coordinates {  (0.02,-3.68) (.04,-3.83)  (2.85,-2.25) (3.1,-2) (2.70,-1.75) } ;

\filldraw[blue, fill opacity=0.10] plot [smooth cycle] coordinates {  (0.02,-2.18) (.04,-2.33)  (2.85,-2.25) (3.1,-2) (2.70,-1.75) } ;

\filldraw[blue, fill opacity=0.10] plot [smooth cycle] coordinates {  (0.02,-.68) (.04,-.83)  (2.85,-2.25) (3.1,-2) (2.70,-1.75) } ;

\filldraw[blue, fill opacity=0.10] plot [smooth cycle] coordinates {  (0.02,0.82) (.04,0.6)  (2.85,-2.25) (3.1,-2) (2.70,-1.75) } ;

\filldraw[blue, fill opacity=0.10] plot [smooth cycle] coordinates {  (0.02,3.68) (.04,3.83)  (2.85,.25) (3.1,0) (2.70,-.1) } ;

\filldraw[blue, fill opacity=0.10] plot [smooth cycle] coordinates {  (0.02,2.18) (.04,2.33)  (2.85,.25) (3.1,0) (2.70,-.1) } ;

\filldraw[blue, fill opacity=0.10] plot [smooth cycle] coordinates {  (0.02,-3.68) (.04,-3.83)  (2.85,-.25)  (3.1,0) (2.70,.1)  } ;

\filldraw[blue, fill opacity=0.10] plot [smooth cycle] coordinates {  (0.02,-2.18) (.04,-2.33)  (2.85,-.25)  (3.1,0) (2.70,.1) } ;

\filldraw[red] (2.8,0)  circle (5pt) (2.8,2)  circle (5pt) (2.8,-2)  circle (5pt);

\filldraw[green] (0,3.75)circle (3pt) (0,2.25)circle (3pt) (0,.75)circle (3pt) (0,-3.75)circle (3pt) (0,-2.25)circle (3pt) (0,-.75)circle (3pt);


\draw  (4,0) node{$\abs{A_i} =r-1$, $d(A_i) =\frac{k-1}{2} $};
\draw  (-.3,0) node[below]{$d(x)=k-1$};
\end{tikzpicture}

\caption{An extremal graph for Theorem~\ref{Not_Star_Theorem} is pictured. Any such graph can be obtained from disjoint copies of a sets of $r+1$ vertices with $k-1$ hyperedges and if $T$ is the balanced double star, possibly a $(k-1,\frac{k-1}{2})$-regular two-sided $r$-uniform hypergraph.}
\label{Extremal}
\end{figure}
\section{Proofs of the main results}

We start with some results about graphs.

\begin{definition} For a Graph $G$, we denote by $d(G)$ the average degree of $G$, that is $d(G) = \frac{2e(G)}{v(G)}.$
\end{definition}

\begin{lemma}\label{folklore} Any non-empty graph $G$ contains a subgraph $G'$ with minimum degree greater than $d(G)/2$.
\end{lemma}
The previous lemma is a well-known result in graph theory, which can be proved using the following lemma.

\begin{lemma}\label{average}
Let $G$ be a graph and $V' \subseteq V$, if $V'$ is incident with at most $\frac{d(G)}{2}\abs{V'}$ edges, then $d(G[V\setminus V']) \geq d(G)$.
\end{lemma}
\begin{proof}
Note that if $m$ is the number of edges incident with $V'$, then we have that \begin{displaymath} 2e(G[V\setminus V']) = 2e(G) - 2m \geq d(G)v(G) - d(G)\abs{V'} = d(G)(\abs{V}-\abs{V'}) = d(G)v(G[V\setminus V']). \qedhere \end{displaymath} \end{proof}

We are going to use the following fact about trees, before proving the next bound on the degrees of the vertices in clusters.

\begin{claim} \label{lowdegree} If $T$ is $k$-edge tree which is not a star, then there exists a vertex of $T$ which is not a leaf and has degree at most $\frac{k+1}{2}$.
\end{claim}
\begin{proof}
Let $T'$ be the tree obtained by $T$ by removing every leaf of $T$, since $T$ is not a star, $T'$ has at least two vertices, take any $v,w$ which are leaves in $T'$, and note that for each, every neighbor but one is a leaf, and also, since at most one the $k$ edges of $T$ is incident with both $u$ and $v$,  we have that $d_T(u)+d_T(v)\leq k+1$.
And so, one of these vertices have the desired properties.
\end{proof}

Now we introduce two more definitions which we will need in the proofs.

\begin{definition}
Let $\h$ be a (multi-)hypergraph. A $(k-1)$-\emph{cluster} is a set of $k-1$ hyperedges of $\h$ that intersect in at least $k-1$ vertices. The intersection of the $k-1$ hyperedges is called the \emph{core} of the $(k-1)$-cluster. The union of the $k-1$ hyperedges is called the \emph{span} of the $(k-1)$-cluster.
\end{definition}

\begin{definition} Let $\h = (V,E)$ be a multi-hypergraph. A multi-hypergraph $\h' = (V',E')$ is called a
\emph{reduced sub-hypergraph} of $\h$ if $V' \subseteq V$ 
and there exists an injection $f:E' \to E$ such that $h \subseteq f(h)$ for all $h\in E'$.
For an edge $h\in E'$ we call $f(h)\in E$ its \emph{correspondent} edge in $\h$.
\end{definition}

In the following claims, we bound the degrees of the vertices in a $(k-1)$-cluster for a hypergraph which does not contain a copy of a Berge tree.

\begin{claim}\label{core}
Let $n,k,r$ be positive integers, with $r \ge k+1$, and let $T$ be a $k$-edge tree. If $\h$ is an $r$-uniform  multi-hypergraph containing no Berge copy of $T$ and $S$ is a $(k-1)$-cluster in $\h$, then the vertices in the core of $S$ have degree exactly $k-1$. In particular, the core vertices of $S$ are only incident with the hyperedges of $S$.
\end{claim}

\begin{proof}
Let $C$ be the set of vertices in the core of $S$.
Suppose, by contradiction, there is a vertex $v$ in $C$ with degree at least $k$, and let $T'$ be a tree obtained from $T$ by removing any two leaves $x,y$. Suppose that the neighbors of these leaves are $x'$ and $y'$ respectively (it is possible that  $x'=y'$).  
Since $C$ has at least $k-1$ vertices and there are $k-1$ hyperedges containing all the vertices in $C$, we can greedily embed $T'$ in $C$ in such a way that $v$ takes the role of $x'$. 
Suppose the vertex $u$ takes the role of $y'$ in this greedy embedding. We can complete the embedding of $T$ by using the last hyperedge of S and an unused vertex in it (one exists since $r\geq k+1$) to embed $y$. 
Then since the degree of $v$ is at least $k$, we have a hyperedge available to embed $x$ as a unused vertex of this hyperedge. 
Thus we have found a Berge copy of $T$ in $\h$, a contradiction.
\end{proof}

\begin{claim}\label{Y}
Let $n,k,r$ be positive integers, with $r \ge k+1$, and let $T$ be a $k$-edge tree which is not a star. If $\h$ is an $r$-uniform multi-hypergraph containing no Berge copy of $T$ and $S$ is a $(k-1)$-cluster of $\h$, then any vertex in the span of $S$ that is incident with a hyperedge not from $S$, has degree at most $\floor{\frac{k-1}{2}}$.
\end{claim}

\begin{proof}
Since $T$ is not a star, by Claim~\ref{lowdegree}, there is a vertex $x \in V(T)$ which is not a leaf and has degree $s$, $s\leq \floor{\frac{k+1}{2}}$, such that all but one of its neighbors is a leaf, let $y$ be the neighbor of $x$ which is not a leaf.  
Suppose, by contradiction, there is a vertex $v$ in the span of $S$ which is incident with a hyperedge that is not in $S$ and $v$ has degree at least $\floor{\frac{k+1}{2}}$. 
Let $C$ be the set of vertices in the core of $S$.
From Claim~\ref{core} we know that $v$ cannot be in $C$.
Pick $s$ hyperedges $h_1,h_2,\dots,h_s$ incident to $v$ in such a way that $h_1$ is not in $S$ and $h_2$ is in $S$. 
Choose a vertex $w \in h_1$ not in $C$ (in fact, every vertex in $h_1$ is outside $C$ by Claim~\ref{core}) and $u \in h_2$ in $C$.  
Choose further distinct vertices $v_3,v_4,\dots,v_s$ from the hyperedges $h_3,h_4,\dots,h_s$.  
The vertex $v$ will be assigned to the vertex  $x$ in the tree, and the vertex $u$ will be assigned to the vertex $y$ ($v_3,v_4,\dots,v_s$ will be assigned to the leaves adjacent to $x$).  
Thus, using the hyperedges $h_1,h_2,\dots,h_s$ we can embed the vertex $x$ and all its neighbors in $T$ using at most $s-1$ hyperedges from $S$ and at most $s-1$ vertices from $C$ ($v$ and $w$ are not in $C$).  

There are at least $(k-1)-(s-1)=k-s$ remaining vertices in $C$.  
Each of these is contained in at least $k-s$ unused hyperedges of $S$.  
Thus, the remaining $k-s$ vertices of the tree can be mapped to distinct vertices from $C$, and the remaining edges of the tree may be assigned to distinct unused hyperedges of $S$.
\end{proof}

\begin{remark}\label{disjoint}
Note that by Claim \ref{core} and Claim \ref{Y}, if $\h$ is a multi-hypergraph with uniformity $r \ge k+1$ that does not contain a Berge copy of a tree on $k$ edges which is not a star, then $(k-1)$-clusters of $\h$ are edge-disjoint.
\end{remark}

\begin{lemma}\label{cluster} Let $k$ be a positive integer and let $T$ be a $k$-edge tree which is not a star.  Let $\h$ be a multi-hypergraph not necessarily uniform, not containing a Berge copy of $T$, and assume that each hyperedge in $\h$ has size at least $k+1$.  If there exists a reduced sub-hypergraph $\h' = (V',E')$  of $\h$ such that $d_{\h'}(v) \geq k-1$ for each $v \in V'$ and $\abs{h} \geq k-1$ for each $h\in E'$, then $\h'$ contains a $(k-1)$-cluster. 
Note that if $S$ is a $(k-1)$-cluster in $\h'$, then the correspondent edges of $S$ in $\h$ are a $(k-1)$-cluster.
\end{lemma}

\begin{proof}
Let $h_2 \in E'$. 
We will show that every vertex in $h_2$ is contained in the same set of hyperedges in $E'$.  
Let $v_1,v_2 \in h_2$, and suppose by contradiction that there exists a hyperedge $h_3$ incident to $v_2$ and not to $v_1$. 
Enumerate the vertices of $T$ by $x_0,x_1,\dots,x_k$ in such a way that the graph induced by the vertices $x_0,x_1,\dots,x_i$ is connected for all $i$, $x_0$ is a leaf of $T$ and $x_0,x_1,x_2,x_3$ is a path of length 3 (such a path exists since $T$ is not a star). For each $i = 1,2,\dots,k$, the vertex $x_i$ is adjacent to exactly one vertex of smaller index, call the edge using $x_i$ and the vertex of smaller index $e_i$. 

We can embed $T$ into $\h$ in the following way. First assign $v_1$ to $x_1$, $h_2$ to $\{x_1,x_2\}$,  $v_2$ to  $x_2$, $h_3$ to $\{x_2,x_3\}$ and any vertex in $v_3 \in h_3\setminus\{v_1,v_2\}$ to $x_3$. For $i=4,\dots,k$, suppose $e_i = \{x_i,x_{j_i}\}$. 
Pick any hyperedge $h_i \in E'$ incident to $v_{j_i}$  and distinct from $h_2,h_3,\dots,h_{i-1}$ (such hyperedges exist since $d_{\h'}(v_{j_i}) \geq k-1$) and assign it to $e_i$.   If $i\leq k-1$, pick any $v_i \in h_i\setminus\{v_1,v_2,\dots,v_{i-1}\}$, and if $i= k$, then let $\tilde{h}_k$ be the correspondent hyperedge of $h_k$ in $\h$. 
As $\tilde{h}_k$ has size bigger than $k$, let $v_{k}$ be any vertex in $\tilde{h}_k\setminus\{v_1,v_2,\dots,v_{k-1}\}$. 
This vertex $v_k$ is assigned to $x_k$.
Finally, since $v_1$ is incident with at least $k-1$ hyperedges distinct to $h_3$, there is a hyperedge $h_1$ incident to $v_1$ and distinct from the already chosen hyperedges.
Let $\tilde{h}_1$ be the correspondent hyperedge of $h_1$.
Take any vertex in $\tilde{h}_1$ which has not been assigned yet and assign it to $x_0$.  
Thus, by replacing the edge $h_i$ with their correspondent hyperedges, we have found a Berge copy of $T$ in $\h$, a contradiction.

It follows that for any $v_1,v_2 \in h_2$, we have that $v_1$ and $v_2$ must be incident with the same set of hyperedges in $\h'$ (by assumption at least $k-1$), and so $\h'$ contains a $(k-1)$-cluster. 
\end{proof}

Lemma \ref{cluster} says that if $\h$ does not contain a Berge copy of a tree and we are able to find a large enough reduced sub-hypergraph, then $\h$ must have at least one $(k-1)$-cluster. The main idea of the proofs of the main results is to show that if $\h$ has too many hyperedges and no Berge copy of a tree, then after removing all $(k-1)$-clusters, we would still be able to find a large enough reduced sub-hypergraph. This would imply that there is still another $(k-1)$-cluster in $\h$, a contradiction.

\begin{proof}[Proof of Theorem \ref{Multi_tree_theorem}]
Let $T$ be a $k$-edge tree, which is not a star.
Suppose that $\h$ is an $n$-vertex $r$-uniform hypergraph  with at least $\frac{n(k-1)}{r}$ hyperedges such that $\h$ does not contain a Berge copy $T$, and let $G$ be the incidence bipartite graph of $\h$, i.e., the bipartite graph with color classes $V(\h)$ and $E(\h)$ where $v\in V(\h)$ is adjacent to $h \in E(\h)$ if and only if $v\in h$. 

Since $\displaystyle e(\h) \ge \frac{n(k-1)}{r}$, we have
$\displaystyle
\frac{e(G)}{v(G)} = \frac{re(\h)}{n + e(\h)} = \frac{r}{\frac{n}{e(\h)}+1} \ge \frac{r}{\frac{r}{k-1}+1} = \frac{r(k-1)}{r+k-1},$ and note that
\begin{displaymath}
\frac{r(k-1)}{r+k-1} \geq k - 2 \Leftrightarrow r(k-1) \geq (k - 2)(r+k-1) = r(k-1) + (k-2)(k-1) - r \Leftrightarrow r \geq (k-2)(k-1).
\end{displaymath}
Hence $d(G) = \frac{2e(G)}{v(G)}\geq 2\left(\frac{r(k-1)}{r+k-1}\right) \geq k-2$, since $r \geq (k-2)(k-1).$ 
Suppose $\h$ has $t$ distinct $(k-1)$-clusters $S_1,S_2,\dots,S_t$ (recall that by Remark~\ref{disjoint} $(k-1)$-clusters are edge-disjoint). 
For each $S_i$, let $X_i$ be the set of vertices which are incident only with hyperedges of $S_i$, let $X = \bigcup_{i=1}^t X_i$ and let $Y$ be the set of vertices that are not in $X$ but are incident with at least one of the $(k-1)$-clusters. 
Let $G_1$ be the induced subgraph of $G$ obtained by removing $X$, $Y$ and all $(k-1)$-cluster hyperedges from the vertex set of $G$. We will show that $d(G_1) \geq d(G)$ (provided $G_1$ is not the empty graph).

The number of edges removed in $G$ is $\sum_{v\in X} d_{\h}(v) + \sum_{v\in Y} d_{\h}(v)$.  Since the degree of each $v\in X$ is at most $k-1$, we have that 
$\displaystyle \left(\sum_{v\in X} d_{\h}(v)\right) \leq \abs{X}(k-1)$. Also $X$ is only incident with the $(k-1)$-cluster hyperedges, thus we also have the bound
$\displaystyle   \left(\sum_{v\in X}d_{\h}(v)\right) \leq tr(k-1)$, and since the degree of each $v\in Y$ is at most $\frac{k-1}{2}$ (Claim~\ref{Y}), we have that $\displaystyle \left(\sum_{v\in Y} d_{\h}(v) \right) \leq \frac{(k-1)\abs{Y}}{2}$. 
Therefore 
\[\left(\sum_{v\in X}d_{\h}(v) + \sum_{v\in Y} d_{\h}(v) \right)(r+k-1)\] 

\[= \left(\sum_{v\in X}d_{\h}(v)\right)r + \left(\sum_{v\in X}d_{\h}(v)\right)(k-1) + \left(\sum_{v\in Y} d_{\h}(v) \right)(r+k-1)\]
\[\leq  \abs{X}r(k-1) + tr(k-1)^2 + \frac{(k-1)\abs{Y}}{2}(r+k-1) \leq r(k-1)(\abs{X} + t(k-1) + \abs{Y}),\]
where in the last inequality we used $\frac{r+k-1}{2} < r$. 
Thus, equality can hold only if $Y = \emptyset$.  

Rearranging we have
\begin{equation}
\label{g}
\left(\sum_{v\in X}d_{\h}(v) + \sum_{v\in Y} d_{\h}(v) \right)\le \frac{r(k-1)}{r+k-1}\left(\abs{X} + t(k-1) + \abs{Y}\right).
\end{equation}
The left-hand side of \eqref{g} is the number of removed edges, and the right-hand side is $d(G)/2$ times the number of removed vertices. 
Therefore, by Lemma~\ref{average}, if $G_1$ is non-empty, we have that  
\[d(G_1) \geq d(G)  \geq 2(k-2).\] 
 Hence, by Lemma~\ref{folklore} there is a subgraph $G_2$ of $G_1$ with minimum degree at least $k-1$. Suppose that $G_2$ has bipartite classes $A \subseteq V(\h)$ and $B \subseteq E(\h)$, and define $\h'$ by taking the vertex set  $V' = A$ and $E' = \{h\cap V': h \in B\}$.
The  condition on the minimum degree of $G_2$ implies that every vertex of $\h'$ has degree at least $k-1$ and every hyperedge of $\h'$ has size at least $k-1$.
Then by Lemma \ref{cluster}, $\h'$ contains a $(k-1)$-cluster, but this $(k-1)$-cluster corresponds to a $(k-1)$-cluster in $\h$ contradicting the fact that we removed every $(k-1)$-cluster from $\h$.
So $\h$ must contain a Berge copy of $T$, unless $G_1$ is empty. 

Note that, for $G_1$ to be empty it is necessary that $d(G) = 2\frac{r(k-1)}{r+k-1}$ and that  equality holds in the inequality~\eqref{g}. 
This can be possible only if $Y = \emptyset$ and 
\[\abs{X} = \frac{1}{k-1}\sum_{v\in X} d_{\h}(v) = tr.\]  
Since every $(k-1)$-cluster contains at least $r$ vertices, we have $\abs{X_i}\geq r$, and so 
each $X_i$ must have size exactly $r$, hence $\h$ is the disjoint union of $t$ hyperedges each with multiplicity $k-1$.  
Therefore the number of vertices would be a multiple of $r$ and $e(\h) = \frac{n(k-1)}{r}$.
Hence if $e(\h) \geq \frac{n(k-1)}{r}$, then $\h$ must contain a Berge copy of $T$, or $r|n$ and $\h$ is the disjoint union of $\frac{n}{r}$ hyperedges each with multiplicity $k-1$.
\end{proof}

\begin{remark}
For $r=(k-2)(k-1)$, the proof above also shows that if $e(\h) > \frac{n(k-1)}{r},$ then $\h$ must contain a Berge copy of $T$. However, the extremal construction does not follow from that proof. \end{remark}

\begin{proof}[Proof of Theorem \ref{Not_Star_Theorem}]
 Let $T$ be a $k$-edge tree which is not a star. We may assume $k>3$, since otherwise $T$ is a path, and we already know the result for paths.
 Let $\h$ be an $n$-vertex hypergraph with at least $\frac{n(k-1)}{r+1}$ hyperedges and $r \ge k(k-2)$. We will proceed by induction on the number of vertices $n$; the base cases $n \le r+1$ are trivial. 

If there is a set $U$ of size $r+1$ which is incident with at most  $k-1$ hyperedges, 
put $V'=V\setminus U$ and let $n' = |V'|=n-r-1$.  
By induction, $\h'$ the hypergraph induced by $V'$, has at most $\frac{n'(k-1)}{r+1}$ hyperedges and equality holds if $r+1|n'$ and $\h'$  is the disjoint union of cliques, unless $T$ is the balanced double star,
then it may contain a $(k-1,\frac{k-1}{2})$-regular two-sided hypergraph as described in the statement of the theorem.
Note that if one of the hyperedges incident with $U$ is incident with  a vertex $v$, $v \in V'$, then $v$ has degree at least $\floor{\frac{k+1}{2}}$, and  $v$ is in a $(k-1)$-cluster of $\h'$, thus we  have a Berge copy of $T$ from Claim~\ref{Y}.
Hence, the $k-1$ hyperedges incident with $U$ are contained in the vertex set $U$ and $\h$ has the desired structure.

Similarly to the proof of Theorem~\ref{Multi_tree_theorem}, we have that
\[\displaystyle\frac{e(G)}{v(G)} = \displaystyle\frac{re(\h)}{n+r(\h)} = \frac{r}{\frac{n}{e(\h)}+1} \ge \frac{r}{\frac{r+1}{k-1} + 1} = \frac{r(k-1)}{r+k},\]  and note that
\begin{displaymath}
\frac{r(k-1)}{r+k} \geq k - 2 \Leftrightarrow r(k-1) \geq (k - 2)(r+k) = r(k-1) + (k-2)k - r \Leftrightarrow r \geq k(k-2).
\end{displaymath}
Hence $d(G) = \frac{2e(G)}{v(G)}\geq 2\left(\frac{r(k-1)}{r+k-1}\right) \geq k-2$, since $r \geq (k-2)(k-1).$
Suppose that $\h$ has $t$ distinct $(k-1)$-clusters $S_1,S_2,\dots, S_t$.  Define the sets $X_1,\dots,X_t,X$ and $Y$ as in the proof of Theorem~\ref{Multi_tree_theorem}.
We are going to remove all vertices and hyperedges of these $(k-1)$-clusters as in the previous proof, and we will denote the incidence bipartite graph of $\h$ by $G$. 
By $G_1$ we will denote the incidence bipartite graph of the hypergraph  $\h'$, obtained from $\h$ after removing the $(k-1)$-clusters.   

If $\abs{X_i} \geq r+1$ for some $i$,  then by taking $U\subseteq X_i$ of size $r+1$, we would have that $U$ is incident with at most $k-1$ hyperedges, and we would be done by induction. 
Hence we assume that $\abs{X_i}\leq r$.

For each $i$, with $\abs{X_i} = r$, we have \[\sum_{v\in X_i} d_{\h}(v) \leq (r-1)(k-1) +1 = \abs{X_i}(k-1)-(k-2),\] since any hyperedge is incident with at most $r-1$ vertices from $X_i$, with the possible exception of at most one hyperedge ($X_i$, if $X_i \in E(\h))$.   

For each $i$, with  $\abs{X_i} \leq r-1$, we have \[\sum_{v\in X_i} d_{\h}(v) \leq \abs{X_i}(k-1) \leq (r-1)(k-1).\] 

Let $a$ be the number of $X_i$, $1 \le i \le t$, with the size $r$. Then we have the following inequalities
\begin{equation}
\label{aaaa}
\displaystyle\sum_{v\in X}d_{\h}(v) = \sum_{\substack{{\abs{X_i}= r} \\ v\in X_i}} d_{\h}(v) + \sum_{\substack{{\abs{X_i} < r} \\ v\in X_i}} d_{\h}(v) \leq  t(r-1)(k-1) + a, 
\end{equation}

and
\begin{equation}
\label{bbbb}
\sum_{v\in X}d_{\h}(v)  \leq \sum_{\substack{{\abs{X_i}= r} \\ v\in X_i}} (\abs{X_i})(k-1)-(k-2)) + \sum_{\substack{{\abs{X_i}< r} \\ v\in X_i}}\abs{X_i}(k-1) =\abs{X}(k-1) - a(k-2).
\end{equation}

We also have 

\begin{equation}
\label{zyx}
tr(k-1) \leq \sum_{v\in X}d_{\h}(v) + \sum_{v \in Y} d_{\h}(v) \leq t(r-1)(k-1) + a + \frac{k-1}{2}\abs{Y},
\end{equation}
where in the first inequality follows from the fact the set the edges   of $t(k-1)$ hyperedges of the $t$ cluster are incident only with the set  $X\cup Y$, and the second inequality follows directly from Claim~\ref{Y} together with the fact that $d_{\h}(v) \leq k-1$ by definition.

Rearranging \eqref{zyx} yields 
\begin{equation}
\label{gh}
\displaystyle t(k-1) \leq a + \frac{\abs{Y}(k-1)}{2}. 
\end{equation}
The following three bounds come from multiplying inequality \eqref{bbbb} and \eqref{aaaa} by $r$ and $k$, respectively, and the bound from Claim \ref{Y} by $k+r$.
\begin{equation}
\label{ab}
    \left(\sum_{v\in X}d_{\h}(v)\right)r  \leq \abs{X}r(k-1) - ar(k-2).
\end{equation}

\begin{equation}
\label{cd}
    \left(\sum_{v\in X}d_{\h}(v)\right)k  \leq t(r-1)k(k-1) + ak.
\end{equation}

\begin{equation}
\label{ef}
    \left(\sum_{v\in Y} d_{\h}(v)\right)(k+r) \leq \frac{\abs{Y}(k-1)}{2}(k+r).
\end{equation}


Now we bound the number of deleted hyperedges times $r+k$. 
From \eqref{ab}, \eqref{cd}, \eqref{ef} and then \eqref{gh}, it follows that
$$\left(\sum_{v\in X}d_{\h}(v) + \sum_{v\in Y} d_{\h}(v)\right)(k+r) \leq \abs{X}r(k-1) - ar(k-2) + t(r-1)k(k-1) + ak + \frac{\abs{Y}(k-1)}{2}(k+r)$$ 
$$= \abs{X}r(k-1) - ar(k-2) + tr(k-1)^2 + t(k-1)(r-k) + ak + \frac{\abs{Y}(k-1)}{2}(k+r)$$
$$\leq  \abs{X}r(k-1) - ar(k-2) + tr(k-1)^2 + (r-k)\left(a + \frac{\abs{Y}(k-1)}{2}\right) + ak + \frac{\abs{Y}(k-1)}{2}(k+r)$$
$$= \abs{X}r(k-1) - ar(k-3) + tr(k-1)^2 + \abs{Y}(k-1)r = r(k-1)(\abs{X}+\abs{Y}+t(k-1)) - ar(k-3)$$ $$\leq  r(k-1)(\abs{X}+\abs{Y}+t(k-1)).$$
Rearranging we have
\begin{equation}
\label{f}
\left(\sum_{v\in X}d_{\h}(v) + \sum_{v\in Y} d_{\h}(v) \right)\le \frac{r(k-1)}{r+k}\left(\abs{X} + t(k-1) + \abs{Y}\right).
\end{equation}
The left-hand side of \eqref{f} is the number of removed edges, and the right-hand side of \eqref{f} is $d(G)/2$ times the number of removed vertices.

Hence, by Lemma~\ref{average} if $G_1$ is nonempty, we have that
\[d(G_1) \geq d(G)  \geq 2(k-2).\] 
Thus, by Lemma~\ref{folklore} we can find a subgraph $G_2$ of $G_1$ with minimum degree at least $k-1$. 
Suppose that $G_2$ has bipartite classes $A \subseteq V$ and $B \subseteq E(\h)$, define $\h'$ by taking the vertex set  $V' = A$ and hyperedge set $E' = \{h\cap V': h \in B\}$.  
The condition on the minimum degree of $G_2$ implies that every vertex of $\h$ has minimum degree at least $k-1$, and every hyperedge of $\h'$ has size at least $k-1$. 
Then by Lemma~\ref{cluster}, $\h'$ contains a $(k-1)$-cluster, which
contradicts that we have removed all $(k-1)$-clusters in $\h$.

For $G_1$ to be empty it is necessary that $d(G) = 2\frac{r(k-1)}{r+k}$, and for \eqref{f} to hold with equality, we must have that $e(\h) = \frac{n(k-1)}{r+1}$.  
To obtain equality in $\eqref{f}$,
it is necessary that $a=0$ (since $k>3$) and that every hyperedge contains one of the $X_i$. 
It then follows that $\abs{X} = t(r-1)$, and by \eqref{gh},  $\abs{Y} = 2t$. 
By \eqref{ef}, for every $v\in Y$, we have $d_{\h}(v) = \frac{k-1}{2}$, so $n = t(r+1)$.  Then $\h$ is a disjoint union of sets of $r+1$ vertices with $k-1$ hyperedges, and a hypergraph 
constructed from the 
classes $A=\{X_1,X_2\dots,X_t\}$ and $B = Y$, where $\{y,X_i\}$ is an edge if $X_i \cup \{y\}$ is a hyperedge of $\h$. 
Note that  $2t =  2\abs{A} = \abs{B}$, the degree of every vertex in $B$ is $\frac{k-1}{2}$ and every vertex of $A$ has degree $k-1$; that is, $\h$ is a $(k-1,\frac{k-1}{2})$-regular two-sided hypergraph.

However, this is only possible if $k$ is odd, and it is simple to check that this construction contains a Berge copy of every $k$-edge tree which is not a balanced $k$-edge double star or the $k$-edge star.   
\end{proof}

\section*{Acknowledgments}

The research of the first three authors was partially supported by the National Research, Development and Innovation Office NKFIH, grants  K116769, K117879 and K126853. The research of the second author is partially supported by  Shota Rustaveli National Science Foundation of Georgia SRNSFG, grant number DI-18-118.  The research of the third author is supported by the Institute for Basic Science (IBS-R029-C1).




\end{document}